\documentclass[11pt]{article}
\usepackage{amsmath,amssymb,enumerate}
\usepackage{amsthm}

\newcommand{\R}{\mathbb{R}}
\newcommand{\C}{\mathbb{C}}

\newtheorem{theorem}{Theorem}[section]
\newtheorem{lemma}[theorem]{Lemma}
\newtheorem{proposition}[theorem]{Proposition}

\theoremstyle{remark}

\theoremstyle{definition}

\newtheorem{example}{Example}
\numberwithin{equation}{section}

\makeatletter
\def\@cite#1#2{[{{\bfseries #1}\if@tempswa , #2\fi}]}
\makeatother                  
\begin{document}

~\vspace{0pt}

\begin{center}
\Large{{\bf
An elementary proof of asymptotic behavior \\[5pt]of solutions of $u''=Vu$
}}
\end{center}

\vspace{5pt}

\begin{center}
Giorgio Metafune%
\footnote{Dipartimento di Matematica ``Ennio De Giorgi'', 
Universit\`a del Salento, Via Per Arnesano, 73100, Lecce, Italy, 
E-mail:\ {\tt giorgio.metafune@unisalento.it}}
and Motohiro Sobajima%
\footnote{
Dipartimento di Matematica ``Ennio De Giorgi'', 
Universit\`a del Salento, Via Per Arnesano, 73100, Lecce, Italy, 
E-mail:\ {\tt msobajima1984@gmail.com}}
\end{center}

\newenvironment{summary}{\vspace{.5\baselineskip}\begin{list}{}{%
     \setlength{\baselineskip}{0.85\baselineskip}
     \setlength{\topsep}{0pt}
     \setlength{\leftmargin}{12mm}
     \setlength{\rightmargin}{12mm}
     \setlength{\listparindent}{0mm}
     \setlength{\itemindent}{\listparindent}
     \setlength{\parsep}{0pt}
     \item\relax}}{\end{list}\vspace{.5\baselineskip}}
\begin{summary}
{\footnotesize {\bf Abstract.}
We provide an elementary proof of the asymptotic behavior of solutions of second order differential equations.}
\end{summary}




\section{Introduction}
The asymptotic behaviour of the solutions of the ordinary differential equation
\begin{equation}\label{ode.V}
u''(x)=V(x)u(x), \qquad x\in (0,\infty)
\end{equation}
is an important tool in various fields of mathematics and mathematical physics, in particular when special functions are involved. 
It can be found in \cite[Section 6.2]{Olver} and partially in \cite[Chapter 10]{BW} and in \cite[Chapter IV]{Erdely}
when $V(x)=f(x)+g(x)$
\begin{equation}\label{ode.f.g}
u''(x)=\bigl(f(x)+g(x)\bigr)u(x), \qquad x\in (0,\infty)
\end{equation}
assuming
\begin{equation}\label{psi}
\psi_{f,g}:=|f|^{-\frac{1}{4}}\left(-\frac{d^2}{dx^2}+g\right)|f|^{-\frac{1}{4}}\in L^1(0,\infty).
\end{equation}
The proof is usually done treating first the cases $f=\pm 1$ and then reducing to them  the general case, 
by the Liouville transformation. We follow the same approach but simplify the cases $f=\pm 1$ 
by using Gronwall's Lemma, instead of successive approximations. 
In order to keep the exposition at an elementary level, 
we avoid also Lebesgue integration and dominated convergence (which could shorten some proofs). 
We consider both the  behavior at infinity and near isolated singularities 
and apply the results to Bessel functions. We also recall that the general case
$$
u''(x)+g(x)u'(x)=V(x)u(x)
$$ 
can be reduced to the form (\ref{ode.V}) (with another $V$)  by writing $u=\frac 12(\exp{\int g})v$.

\section{Behavior near infinity in the simplest cases} 
\; First we consider the cases $f\equiv 1$ and $f\equiv -1$ and we prove the following results to which the general case reduces. 

\begin{proposition}\label{lem2.1}
If $f=1$, $g\in L^1(0,\infty)$, then there exist 
two solutions $u_1$ and $u_2$ of 
\eqref{ode.f.g} such that, as $x\to \infty$, 
\begin{gather}
\label{bhv.ex}
e^{-x}u_1(x)\to 1, \qquad e^{-x}u_1'(x)\to 1, 
\\
\label{bhv.e-x}
e^{x}u_2(x)\to 1, \qquad e^{x}u_2'(x)\to -1.
\end{gather}

\end{proposition}

\begin{proposition}\label{lem2.2}
If $f=-1$, $g\in L^1(0,\infty)$, then there exist two solutions $v_1$ and $v_2$ of 
\eqref{ode.f.g}  such that, as $x\to \infty$, 
\begin{gather}
\label{bhv.eix}
e^{-ix}u_1(x)\to 1, \qquad e^{-ix}u_1'(x)\to i, 
\\
\label{bhv.e-ix}
e^{ix}u_2(x)\to 1, \qquad e^{ix}u_2'(x)\to -i.
\end{gather}
\end{proposition}

\noindent By variation of parameters, every solution of (\ref{ode.f.g}) can be written as 
\begin{equation} \label{variation}
u(x)=c_1e^{\zeta x}+c_2 e^{-\zeta x}+\frac{1}{2\zeta}\int_{a}^x(e^{\zeta (x-s)}-e^{-\zeta (x-s)})g(s)u(s)\,ds,
\quad x\in [a,\infty), 
\end{equation}
with $c_1, c_2 \in \C$, $\zeta=1,i, -i$ and $a>0$. In the following Lemma we choose $c_1=1, c_2=0$ 
to construct a solution which behaves like $e^{\zeta x}$ as $x \to \infty$, $\zeta=1,i,-i$.

\begin{lemma}\label{bdd}
Let $\zeta\in \{1,i,-i\}$, $a>0$ and $g\in L^1(a,\infty)$. If $u\in C^2([a,\infty))$ satisfies 
\[
u(x)=e^{\zeta x}+\frac{1}{2\zeta}\int_{a}^x(e^{\zeta (x-s)}-e^{-\zeta (x-s)})g(s)u(s)\,ds,
\qquad x\in [a,\infty), 
\]
then $z(x):=e^{-\zeta x}u(x)$ satisfies 
\begin{align}
\label{z.est}
|z(x)|&\leq e^{\int_{a}^x|g(r)|\,dr}, \qquad x\in [a,\infty)
\\
\label{hz}
\|zg\|_{L^1(a,\infty)}
&\leq e^{\|g\|_{L^1(a,\infty)}}-1.
\end{align}
\end{lemma}

\begin{proof}
Note that  
\[
z(x)=1+\frac{1}{2\zeta}\int_{a}^x(1-e^{-2\zeta (x-s)})g(s)z(s)\,ds, \quad x\in [a,\infty).
\]
Since $|1-e^{-2\zeta (x-s)}|\leq 2$ for $s \leq x$, we see that for $x\geq a$, 
\[
|z(x)|\leq 1+\left|
  \frac{1}{2\zeta}\int_{a}^x(1-e^{-2\zeta (x-s)})g(s)z(s)\,ds
\right|
\leq 
  1+\int_{a}^x|g(s)|\,|z(s)|\,ds.
\]
Thus Gronwall's lemma implies \eqref{z.est}, in particular $z$ is bounded on $[a,\infty)$ and then 
$zg\in L^1(a,\infty)$. 
Moreover we have 
\begin{align*}
\|zg\|_{L^1(a,\infty)}
\leq
\int_{a}^\infty |g(s)|\,e^{\int_{a}^s|g(r)|\,dr}\,ds
=e^{\|g\|_{L^1(a,\infty)}}-1.
\end{align*}
\end{proof}

\begin{proof}[Proof of Proposition \ref{lem2.1}] 
Let $a>0$ such that $\|g\|_{L^1(a,\infty)}<\log 2$ 
and let $u$ be in Lemma \ref{bdd} with $\zeta=1$. 
Then $u$ is one solution of \eqref{ode.f.g} with $f=1$. 
Set $z(x)=e^{-x}u(x)$. Then noting that as $x\to \infty$, 
\begin{align*}
\left|\int_a^x\,e^{-2(x-s)}g(s)z(s)\,ds\right|
&\leq
\int_a^{\frac{a+x}{2}}\,e^{-2(x-s)}|g(s)z(s)|\,ds
+\int_{\frac{a+x}{2}}^x\,|g(s)z(s)|\,ds
\\
&\leq e^{-x+a}\|gz\|_{L^1(a,\infty)}
+\|gz\|_{L^1(\frac{a+x}{2},\infty)}\to 0, 
\end{align*}
we see that $z$ satisfies 
\begin{align*}
z(x)
&\to z_\infty:=1+\int_{a}^\infty g(s)z(s)\,ds\quad \text{as}\ x\to \infty,
\\
z'(x)&
=\int_{a}^xe^{-2(x-s)}g(s)z(s)\,ds
\to 0\quad \text{as}\ x\to \infty.
\end{align*}
By \eqref{hz}, we deduce that $\|zg\|_{L^1(a,\infty)}<1$. 
Therefore $|z_{\infty}-1| \leq \|zg\|_{L^1(a,\infty)}<1$ and hence $z_\infty \neq 0$.  
The function $u_1(x):=z_{\infty}^{-1}e^{x}z(x)$ 
satisfies \eqref{bhv.ex}. 
Moreover, since $u_1^{-2}$ is integrable near $\infty$, 
another solution of \eqref{ode.f.g} 
is given by
\begin{equation} \label{u2}
u_2(x)=2u_1(x)\int_{x}^\infty\frac{1}{u_1(s)^2}\,ds.
\end{equation}
Integrating by parts we deduce that, as $x\to \infty$, 
\begin{align*}
e^{x}u_2(x)
&=2z_{\infty}e^{2x}z(x)\int_{x}^\infty\frac{1}{e^{2s}[z(s)]^2}\,ds
\\
&=z_{\infty}e^{2x}z(x)\left(
-\left[\frac{1}{e^{2s}[z(s)]^2}\right]_{s=x}^{s=\infty}
-2\int_{x}^\infty\frac{z'(s)}{e^{2s}[z(s)]^3}\,ds\right)
\to 1
\end{align*}
and  
\begin{align*}
[e^{x}u_2(x)]'
&=2z_{\infty}e^{2x}z'(x)\int_{x}^\infty\frac{1}{e^{2s}[z(s)]^2}\,ds
+
2e^{x}u_2(x)-\frac{2z_{\infty}}{z(x)}
\to 0.
\end{align*}
\end{proof}

\begin{proof}[Proof of Proposition \ref{lem2.2}] 
Let $a>0$ such that $\|g\|_{L^1(a,\infty)}<\log 2$ and 
let $\tilde{u}_1$ and $\tilde{u}_2$ be as in Lemma \ref{bdd} with 
$\zeta=i$ and with $\zeta=-i$, respectively.
Noting that both $\tilde{u}_1$ and $\tilde{u}_2$ satisfy \eqref{ode.f.g} with $f=-1$, 
and setting $z_1(x)=e^{-ix}\tilde{u}_1(x)$ and $z_2(x)=e^{ix}\tilde{u}_2(x)$, we have as $x \to \infty$
\begin{align*}
e^{2i x}\left(z_1(x)-1-\frac{1}{2i}\int_{a}^\infty g(s)z_1(s)\,ds\right)
&\to 
\frac{1}{2i}\int_{a}^\infty e^{2is}g(s)z_1(s)\,ds, 
\\
e^{-2ix}\left(z_2(x)-1+\frac{1}{2i}\int_{a}^\infty g(s)z_2(s)\,ds\right)
&\to 
-\frac{1}{2i}\int_{a}^\infty e^{-2is}g(s)z_2(s)\,ds
\end{align*}
and
$$
e^{2 ix}z'_1(x) \to \int_a^\infty e^{2is} g(s)z_1(s) \, ds, 
\qquad e^{-2 ix}z'_2(x) \to \int_a^\infty e^{-2is} g(s)z_2(s) \, ds.$$
It follows that 
$\tilde{u}_1 \approx \xi_1e^{ix}+\xi_2e^{-ix}$, $\tilde{u}'_1 \approx i\xi_1e^{ix}-i\xi_2e^{-ix}$ and $\tilde{u}_2\approx \eta_1e^{ix}+\eta_2e^{-ix}$, $\tilde{u}'_2\approx i\eta_1e^{ix}-i\eta_2e^{-ix}$ as $x \to \infty$ where  
$$\xi_1= 1+\frac{1}{2i} \int_a^\infty g(s)z_1(s)\, ds, \qquad \xi_2= -\frac{1}{2i} \int_a^\infty e^{2is} g(s)z_1(s)\, ds,$$
and similarly for $\eta_1, \eta_2$. From  \eqref{hz} we see that
$|\xi_1|>1/2$, $|\xi_2|<1/2$, $|\eta_1|<1/2$ and $|\eta_2|>1/2$ and hence
$|\xi_1\eta_2-\xi_2\eta_1|>0$ and 
$\tilde{u}_1$ and $\tilde{u}_2$ are linearly independent. 
Therefore we can construct solutions $u_1$ and $u_2$  
which satisfy \eqref{bhv.eix} and \eqref{bhv.e-ix}, respectively.
\end{proof}

\noindent We consider now the case $f=0$, assuming extra conditions on $g$.

\begin{proposition}\label{f.zero}
Assume that $xg\in L^1(0,\infty)$. 
Then there exist two solutions $u_1$ and $u_2$ of 
\begin{align}\label{ode.f.zero}
u''(x)=g(x)u(x)
\end{align}
such that
\begin{gather*}
x^{-1}u_1(x)\to 1, 
\qquad
u_1'(x)\to 1, 
\\
u_2(x)\to 1, \qquad 
xu_2'(x)\to 0
\end{gather*}
as $x\to \infty$, respectively.
\end{proposition}

\begin{proof}
Set $u(x):=xz(x)$. Then $z''+(2/x)z'=gz$ and, assuming $z'(a)=0$ we obtain
\begin{equation} \label{z'1}
z'(x)=x^{-2}\int_a^x s^2g(s)z(s)\, ds.
\end{equation}
Then assuming $z(a)=1$
\begin{align} \label{z1}
\nonumber 
|z(x)-1|
&\le\int_b^xt^{-2}\left(\int_a^t s^2|g(s)z(s)|\, ds\right)\, dt
\\&=\int_a^x \left(\int_s^x t^{-2}\, dt\right)s^2|g(s)z(s)|\, ds
\le \int_a^xs|g(s)z(s)|\, ds.
\end{align}
Gronwall's lemma yields 
$$|z(x)| \le e^{\int_a^x s|g(s)|\, ds}$$
hence $z$ is bounded and $z' \in L^1(a, \infty)$ by (\ref{z'1}). 
As in the proof of Proposition \ref{lem2.1}, 
$z(x) \to z_\infty \neq 0$ if $a$ is sufficiently large. 
Moreover, since as $x\to \infty$, 
\begin{align*}
|xz'(x)|\leq 
\sqrt{\frac{a}{x}}\int_a^{\sqrt{ax}} s|g(s)z(s)|\, ds
+\int_{\sqrt{ax}}^x s|g(s)z(s)|\, ds\to 0,
\end{align*}
$u_1(x):=z_\infty^{-1}xz(x)$ satisfies the statement. 
Another solution $u_2$ of \eqref{ode.f.g} is given by
\begin{equation*}
u_2(x):=u_1(x)\int_{x}^\infty\frac{1}{u_1(s)^2}\,ds.
\end{equation*}
As in the proof of Proposition \ref{f.pos} we can verify that 
$u_2$ satisfies $u_2(x)\to 1$ and $xu_2'(x)\to 0$ as $x\to \infty$.
\end{proof}

\noindent Observe the integrability condition for $xg$ near $\infty$ is necessary. In fact, if $g(x)=cx^{-2}$ the above equation has solutions $x^{\alpha}$ if $\alpha^2-\alpha=c$.

\section{Behavior near infinity in the general case}

We recall that the function $\psi_{f,g}$ is defined in (\ref{psi}) 
and set $v_j(x)=|f|^{1/4}u_j(x)$, $j=1,2$ if $u_1, u_2$ are solutions of (\ref{ode.f.g}). 
The hypothesis $|f|^{1/2}$ not summable near $\infty$ guarantees that 
the Liouville transformation $\Phi$  of Lemma \ref{change} maps $(a,\infty)$ onto $(0,\infty)$, 
so that the results of the previous section apply. 
When it is not satisfied $\Phi$ maps $(a,\infty)$ onto a bounded interval $(0,b)$ and the behavior 
of the solutions of (\ref{ode.phi}) near $b$ is more elementary 
(in some cases one can use Proposition \ref{f.zero}).

\begin{proposition}\label{f.pos}
Assume that $f(x)>0$ in $(a,\infty)$, $|f|^{1/2} \not \in L^1(a, \infty)$ 
and $\psi_{f,g}\in L^1(a,\infty)$. 
Then there exist two solutions $u_1$ and $u_2$ of \eqref{ode.f.g} 
such that as $x \to \infty$
\begin{gather}
\label{bhv.f.ex}
e^{-\int_{a}^x|f(r)|^{1/2}dr}v_1(x)\to 1, 
\qquad
|f(x)|^{-1/2}e^{-\int_{a}^x|f(r)|^{1/2}dr} v_1'(x)\to 1, 
\\
\label{bhv.f.e-x}
e^{\int_{a}^x|f(r)|^{1/2}dr}v_2(x)\to 1, \qquad 
|f(x)|^{-1/2}e^{\int_{a}^x|f(r)|^{1/2}dr} v_2'(x)\to -1.
\end{gather}

\end{proposition}

\begin{proposition}\label{f.neg}
Assume that $f(x)<0$ in $(a,\infty)$, $|f|^{1/2} \not \in L^1(a, \infty)$ 
and $\psi_{f,g}\in L^1(a,\infty)$. 
Then there exists two solutions $u_1$ and $u_2$ of 
\eqref{ode.f.g} 
such that as$x \to \infty$
\begin{gather}
\label{bhv.f.eix}
e^{-i\int_{a}^x|f(r)|^{1/2}dr}v_1(x)\to 1, 
\qquad
|f(x)|^{-1/2}e^{-i\int_{a}^x|f(r)|^{1/2}dr} v_1'(x)\to i, 
\\
\label{bhv.f.e-ix}
e^{i\int_{a}^x|f(r)|^{1/2}dr}v_2(x)\to 1, \qquad 
|f(x)|^{-1/2}e^{i\int_{a}^x|f(r)|^{1/2}dr} v_2'(x)\to -i.
\end{gather}
\end{proposition}

\noindent 
The proof is based on the well-known Liouville transformation that we recall below.
\begin{lemma}\label{change}
Let $a>0$ and assume that $f\in C^2([a, \infty))$ satisfies 
$|f(x)|>0$, $|f|^{1/2} \not \in L^1(a, \infty)$. 
Define $\Phi\in C^2([a,\infty))$ by
\[\Phi(x):=\int_{a}^x|f(r)|^{1/2}\,dr, \quad x\in[a,\infty).\]
Then $\Phi^{-1}:[0,\infty)\to [a,\infty)$ and if
$u$ satisfies \eqref{ode.f.g} the function 
\[
w(y):= |f(\Phi^{-1}(y))|^{1/4}u(\Phi^{-1}(y)), \quad y\in [0,\infty)
\]
satisfies 
\begin{equation}\label{ode.phi}
w''(y)=
\left(
  \frac{f(\Phi^{-1}(y))}{|f(\Phi^{-1}(y))|}
+
\frac{\psi_{f,g}(\Phi^{-1}(y))}{|f(\Phi^{-1}(y))|^{1/2}}\right)w(y).
\end{equation}
\end{lemma}

\begin{proof}
Note that 
$\Phi'(x)=|f(x)|^{1/2}$ and 
$\frac{d(\Phi^{-1})}{dy}(y)=|f(\Phi^{-1}(y))|^{-1/2}$.
Setting $w(y)=|f(\Phi^{-1}(y))|^{1/4}u(\Phi^{-1}(y))$ 
(and using $\xi=\Phi^{-1}(y)$ for simplicity), 
we have
\begin{align*}
w'(y)
&=\frac{d}{dx}\left[|f|^{1/4}u\right](\xi)\frac{d(\Phi^{-1})}{dy}(y)
\\
&=
|f(\xi)|^{-1/4}u'(\xi)+\left[|f|^{-1/2}\frac{d}{dx}|f|^{1/4}\right](\xi)u(\xi)
\\
&=
\left[|f|^{-1/4}u'-\frac{d}{dx}(|f|^{-1/4})u\right](\xi),
\\
w''(y)
&=
\frac{d}{dx}
\left[|f|^{-1/4}u'-\frac{d}{dx}(|f|^{-1/4})u\right](\xi)
\frac{d(\Phi^{-1})}{dy}(y)
\\
&=
|f(\xi)|^{-3/4}u''(\xi)-
\left[|f|^{-1/2}\frac{d^2}{dx^2}|f|^{-1/4}\right](\xi)u(\xi)
\\
&=
|f(\xi)|^{-1}
(f(\xi)+g(\xi))w(y)
-
\left[|f|^{-3/4}\frac{d^2}{dx^2}|f|^{-1/4}\right](\xi)w(y).
\end{align*}
Thus we obtain \eqref{ode.phi}.
\end{proof}

\begin{proof}[Proof of Propositions \ref{f.pos} and \ref{f.neg}]
It suffices to apply Propositions \ref{lem2.1} and \ref{lem2.2} 
to the respective cases $f>0$ and $f<0$. 
Set $h(y)=\psi_{f,g}(\Phi^{-1}(y))|f(\Phi^{-1}(y))|^{-1/2}$. 
Then 
\[
\int_{0}^{b}
|h(y)|\,dy
=
\int_{a}^{\infty}
|\psi_{f,g}(x)|\,dx.
\]
Therefore Propositions \ref{lem2.1} and \ref{lem2.2} are applicable 
to $w''=\pm w+h w$, respectively.
Finally, using Lemma \ref{change} and taking $u(x)=|f(x)|^{-1/4}w(\Phi(x))$, 
we obtain the respective assertions in 
Propositions \ref{f.pos} and \ref{f.neg}.
\end{proof}

\section{Behavior near interior singularities} 
If $f$ and $g$ have local singularities at $x_0$,  
then the behavior of solutions  near $x_0$ is also considerable.
For simplicity, we take $x_0=0$. The following propositions 
are meaningful when $|f|^{1/2}$ is not integrable near $0$, in particular when $|f|^{1/2}=cx^{-1}$. We recall that $v_j(x)=|f(x)|^{1/4}u_j(x)$, $j=1, 2$.

\begin{proposition}\label{f.pos.sing}
Assume that $f(x)>0$ in $(0,\infty)$ and $\psi_{f,g}\in L^1(0,\infty)$. 
Then there exist two solutions $u_1$ and $u_2$ of \eqref{ode.f.g} such that as $x\downarrow 0$
\begin{gather*}
e^{-\int_x^1|f(r)|^{1/2}dr}v_1(x)\to 1, 
\qquad
|f(x)|^{-1/2}e^{-\int_x^1|f(r)|^{1/2}dr} v_1'(x)\to -1, 
\\
e^{\int_x^1|f(r)|^{1/2}dr}v_2(x)\to 1, 
\qquad
|f(x)|^{-1/2}e^{\int_x^1|f(r)|^{1/2}dr} v_2'(x)\to 1.
\end{gather*}

\end{proposition}

\begin{proposition}\label{f.neg.sing}
Assume that $f(x)<0$ in $(0,\infty)$ and $\psi_{f,g}\in L^1(0,\infty)$. 
Then there exist two solutions $u_1$ and $u_2$ of \eqref{ode.f.g} 
such that as $x\downarrow 0$
\begin{gather*}
e^{-\int_x^1|f(r)|^{1/2}dr}v_1(x)\to 1, 
\qquad
|f(x)|^{-1/2}e^{-\int_x^1|f(r)|^{1/2}dr} v_1'(x)\to -i, 
\\
e^{\int_x^1|f(r)|^{1/2}dr}v_2(x)\to 1, 
\qquad
|f(x)|^{-1/2}e^{\int_x^1|f(r)|^{1/2}dr} v_2'(x)\to i.
\end{gather*}
\end{proposition}

\begin{proof}[Proof of Propositions \ref{f.pos.sing} and \ref{f.neg.sing}]
Setting $w(s):=s u(s^{-1})$ we see that 
\begin{align*}
w''(s)&=s^{-3}u''(s^{-1})
\\
&=s^{-3}(f(s^{-1})+g(s^{-1}))u(s^{-1})
=s^{-4}(f(s^{-1})+g(s^{-1}))w(s).
\end{align*}
Let $\tilde{f}(s):=s^{-4}f(s^{-1})$ 
and $\tilde{g}(s):=s^{-4}g(s^{-1})$.
Noting that 
\begin{align*}
\psi_{\tilde{f},\tilde{g}}(s)
&=s|f(s^{-1})|^{-1/4}
\left(-\frac{d^2}{ds^2}+s^{-4}g(s^{-1})\right)\left(s|f(s^{-1})|^{-1/4}\right)
\\
&=s^{-2}|f(s^{-1})|^{-1/4}
\left(-\frac{d^2}{dx^2}|f|^{-1/4}+g|f|^{-1/4}\right)(s^{-1})
\\
&=s^{-2}\psi_{f,g}(s^{-1}),
\end{align*}
we have $\psi_{\tilde{f},\tilde{g}}\in L^1((0,\infty))$, and hence 
Propositions \ref{f.pos} and \ref{f.neg} can be applied. 
Since 
\[
\int_{1}^{s}|\tilde{f}(r)|^{1/2}dr=\int_{1/s}^1|f(t)|^{1/2}dt,
\]
we obtain the respective assertions in Propositions \ref{f.pos.sing} and \ref{f.neg.sing}.
\end{proof}

\section{Examples from special functions}
Some examples
illustrate the application of the results of the previous sections.

\begin{example}[Modified Bessel functions]\label{ex.mbessel}
We consider the modified Bessel equation of order $\nu$
\begin{align}\label{mbessel.eq}
u''+\frac{u'}{r}-\left(1+\frac{\nu^2}{r^2}\right)u=0,
\end{align}
All solutions of \eqref{mbessel.eq} can be 
written through the modified Bessel functions $I_{\nu}$ and $K_{\nu}$. 
Both $I_{\nu}$ and $K_{\nu}$ are positive, 
$I_{\nu}$ is monotone increasing and $K_\nu$ is monotone decreasing 
(see e.g., \cite[Theorem 7.8.1]{Olver}).
Proposition \ref{lem2.1} and Proposition \ref{f.pos.sing} 
give the precise behavior of $I_\nu$ and $K_\nu$ 
near $\infty$ and near $0$, respectively. 
In fact, \eqref{mbessel.eq} can be written as
\begin{align}\label{mbessel.eq2}
(\sqrt{r}u)''=\left(1+\frac{4\nu^2-1}{4r^2}\right)(\sqrt{r}u).
\end{align}
Since $1/r^2$ is integrable near $\infty$, 
choosing $f=1$ and $g=\frac{4\nu^2-1}{4r^2}$, we see from Proposition \ref{lem2.1} that
\begin{gather*}
\sqrt{r}e^{-r}I_\nu(r)\to c_1 \neq 0 \quad\text{and}\quad \sqrt{r}e^{r}K_\nu(r)\to c_2 \neq 0 
\qquad\text{as $r\to \infty$}.
\end{gather*}
Moreover, if $\nu\neq 0$, then 
choosing $f(r)=\frac{\nu^2}{r^2}$ and $g(r)=1-\frac{1}{4r^2}$, that is, 
$\psi_{f,g}(r)=r/\nu$, 
from Proposition \ref{f.pos.sing} we have 
\begin{gather*}
r^{-\nu}I_\nu(r)\to c_3 \neq 0 \quad {\rm and }\quad r^{\nu}K_\nu(r)\to c_4 \neq 0
\qquad\text{as $r\downarrow 0$}.
\end{gather*}
If $\nu=0$, then putting 
$w(s)=u(e^{-s})$ we obtain
\[
w''(s)=e^{-2s}w(s), \qquad s\in \R.
\]
Therefore using Proposition \ref{f.zero} with $\tilde{g}(s)=e^{-2s}$ 
and taking $u(x)=w(-\log x)$, we have 
\begin{gather*}
I_0(r)\to c_5 \neq 0\quad\text{and}\quad 
|\log r|^{-1} K_0(r)\to c_6 \neq 0
\qquad\text{as $r \downarrow 0$}.
\end{gather*}
\end{example}

\begin{example}[Fundamental solution of $\lambda-\Delta$] 
For $n \geq 3$, $\lambda \geq 0$ the fundamental solution $v_\lambda$ of $\lambda-\Delta$ 
can be computed by integrating the heat kernel:
\begin{equation*}
v_\lambda (r)=\int_0^\infty \frac{1}{(4\pi t)^{n/2}}e^{-\lambda t-\frac{r^2}{4t}}\, dt,
\end{equation*}
where $r=|x|$. 
Clearly $v_\lambda(r) \le v_0(r)=cr^{2-n}$, $v_\lambda (r) \to 0$ as $r \to \infty$. 
The function $v=v_\lambda$ satisfies
\begin{equation*}
v''+\frac{n-1}{r}v'=\lambda v
\end{equation*} 
or, setting $v=r^{(1-n)/2}w$, 
\begin{equation*}
w''=\left(\lambda+\frac{n^2-1}{4r^2}\right)w.
\end{equation*} 
Proceeding as in the example above we see that $r^{2-n}v(r) \to c_1 \neq 0$ as $r \to 0$ 
and $r^{(n-1)/2}e^{\sqrt{\lambda}r}v(r)\to c_2 \neq 0$ as $r \to \infty$.
\end{example}

\begin{example}[Bessel functions]\label{ex.bessel}
Next we consider the Bessel equation of order $\nu$
\begin{align}\label{bessel.eq}
u''+\frac{u'}{r}+\left(1-\frac{\nu^2}{r^2}\right)u=0, 
\end{align}
or equivalently, 
\begin{align*}
(\sqrt{r}u)''=\left(-1+\frac{4\nu^2-1}{4r^2}\right)(\sqrt{r}u).
\end{align*}
All solutions of \eqref{bessel.eq} can be 
written through the Bessel functions $J_{\nu}$ and $Y_{\nu}$. 
As in Example \ref{ex.mbessel}, 
from Propositions \ref{f.pos.sing} (for $\nu>0$) and \ref{f.zero} (for $\nu=0$)
we obtain the behavior of $J_{\nu}$ and $Y_\nu$ near $0$ 
\begin{equation*} 
  r^{-\nu} J_\nu(r)\to c_1\neq 0, \quad{\rm and}\quad
  r^{\nu}  Y_\nu(r)\to c_2\neq 0 \qquad \text{as $r\downarrow 0$}
\end{equation*} 
and if $\nu=0$, 
\begin{equation*} 
  |\log r| J_0(r)\to c_3\neq 0, \quad{\rm and}\quad
  Y_0(r)\to c_4\neq 0 \qquad \text{as $r\downarrow 0$}.
\end{equation*} 
In view of Proposition \ref{lem2.2} 
the behavior of $J_\nu$ and $Y_\mu$ near $\infty$ is given by 
\begin{equation*} 
  |\sqrt{r} J_\nu(r)-c_5\cos(r+\theta_1)| \to 0, \quad{\rm and}\quad
  |\sqrt{r} Y_\nu(r)-c_6\cos(r+\theta_2)| \to 0,
\end{equation*} 
as $r\to \infty$, where $c_5\neq0$, $c_6\neq 0$ and $\theta_1,\theta_2\in [0,\pi)$ satisfy $\theta_1\neq\theta_2$.
\end{example}


\begin{thebibliography}{99}

  \bibitem{BW}\label{BW}
     R.\ Beals, R.\ Wong, 
       ``Special functions,'' (A graduate text)
        Cambridge Studies in Advanced Mathematics {\bf 126}. 
        Cambridge University Press, Cambridge, 2010.

\bibitem{Erdely} \label{Erdely}
A. Erdely, ``Asymptotic expansions," Dover 1956.


  \bibitem{Olver}\label{Olver}
     F.W.J.\ Olver, 
       ``Asymptotics and special functions,'' 
      Computer Science and Applied Mathematics, 
      Academic Press, New York-London, 1974. 

\end{thebibliography}
\end{document}